\title{Convex bounds for last passage percolation with dependent identically distributed weights}
\date{April 12, 2026}

\documentclass{article}
\usepackage{graphicx}
\usepackage{subfig}
\usepackage{multirow}
\usepackage{amsfonts}
\usepackage{xcolor}
\usepackage{amsmath}
\usepackage{tikz}

\newtheorem{theorem}{Theorem}

\begin{document}

\author {
Isaac Meilijson \\
{\em School of Mathematical Sciences} \\
{\em Tel Aviv University, Tel Aviv, Israel} \\
{\em E-mail: \tt{isaco@tauex.tau.ac.il}} \\
\\
}
\maketitle

\pagenumbering{arabic}

\begin{abstract}

On the $Z^2$ lattice, vertices are assigned random weights $W(i,j)$. The point-to-point last passage percolation (LPP) time $S_{M,N+1-M}$ between $(1,1)$ and $(M,N+1-M)$ is the maximum total weight among all upward/right-oriented paths connecting the two. Point-to-line LPP time $R_N$ is the maximum of these maximal total weights over $M$. Asymptotic distributions and fluctuations of these LPP times have been studied for i.i.d. weights. The current study deals with identically distributed but not necessarily independent weights, and maximizes LPP times in the sense of increasing convex dominance. In particular, maximal expected LPP times are identified, in the class of all weight couplings with a given marginal distribution. For the case of mean-$1$ exponentially distributed weights, there is a coupling for which $R_N$ is the shifted exponential variable $R_N^* = N W(1,1) + \log(N!)$, such that $E[\Psi(R_N)] \le E[\Psi(R_N^*)]$ for all couplings and all convex non-decreasing functions $\Psi$ for which these expectations are well defined. In contrast to ${{R_N^*} \over N}= W(1,1)+{{\log(N!)} \over N}$, with variance $1$ and mean diverging to $\infty$ like $\log(N)$, ${{R_N} \over N}$ converges a.s. to $2$ for the commonly studied i.i.d. weights. As for {\em small} LPP, expected LPP time is at least $NE[W(1,1)]$, attained by assigning to each anti-diagonal identical weights. The minimal possible variance of $R_N$ is asymptotically zero for exponential weights.

\end{abstract}

\section{Introduction} \label{techintro}

On the $Z^2$ lattice, vertices are assigned random weights $W(i,j)$ with a common distribution $F$. The point-to-point last passage percolation (LPP) time $S_{M,N+1-M}$ between $(1,1)$ and $(M,N+1-M)$ is the maximum total weight among all upward/right-oriented paths connecting the two. Point-to-line LPP time $R_N$ is the maximum of these maximal total weights over $M$.

\medskip

The asymptotic behavior (as $N \rightarrow \infty$) of ${R_N \over N}$ and of ${{S_{N,M}} \over {N+M}}$ (keeping $\gamma={N \over {N+M}}$ fixed) has been studied for the i.i.d. model. These variables have a.s. constant limits, as proved by means of Kingman's sub-additive ergodic theorem (see the survey by Blair-Stahn \cite{blairstahn}).
Explicit answers have been provided for exponential and geometric $F$. Among these, for $Exp(1)$-distributed weights:
Johansson \cite{Johansson} identified the asymptotic Tracy-Widom distribution of point-to-point normalized LPP ${(S_{N,M}-2N)/N^{1 \over 3}}$. Within their work on the Kardar-Parisi-Zhang (KPZ) universality class, Borodin et. al. \cite{Borodinetal} and Sasamoto \cite{Sasamoto} identified the asymptotic Tracy-Widom distribution of
the point-to-line normalized LPP ${(R_N-2N)/N^{1 \over 3}}$.  The limiting shape function of ${{S_{N,M}} \over {N+M}}$ (as a function of $\gamma$) for point-to-point LPP was discovered by R\"ost \cite{Rost}.

\medskip

The current study deals with identically distributed but not necessarily independent weights, and maximizes LPP times in the sense of increasing convex dominance. Joint distributions (or {\em couplings}) of the weights $W(i,j)$ will be introduced, with identical $F$ marginals, such that their LPP times $R_N=R_N^*$ and $S_{N,M}=S_{N,M}^*$ convexly dominate the corresponding LPP times for any coupling.
In particular, maximal expected LPP times are identified, in the class of all weight couplings with a given marginal distribution.

For $F$ exponential, Pareto and uniform, the worst case point-to-point and point-to-line LPP times $R_N^*$ and $S_{N,M}^*$ are linear functions of $W(1,1)$, for every $N$ and $M$. E.g., $R_N^*=N W(1,1) + \log(N!)$ for $Exp(1)$ $F$.

\section{Preliminaries} \label{techprelim}

{\bf A menu of oriented $N$-partite graphs on ${\mathbb Z}^2$}. For $N \ge 1$, $C_N$ is the graph on ${\mathbb Z}^2$ with ${{N(N+1)} \over 2}$ vertices, the sites $(i,j)$, for $1 \le i,j \le N$ and $i+j \le N+1$, forming a triangle. This triangle is partitioned naturally by the values of $i+j$ into $N$ disjoint {\em anti-diagonals}. All graphs to be considered are oriented $N$-partite graphs in which only vertices in adjacent anti-diagonals are connected, with edges going only from an anti-diagonal to the anti-diagonal with one more vertex. $C_N$ is the {\em complete} $N$-partite graph. Its {\em point-to-line} sub-graph $L_N$ (L for {\em line}) has $(i,j)$ on the $i+j-1$ anti-diagonal connected only to the two ("up" and "right") vertices $(i,j+1)$ and $(i+1,j)$ on the $i+j$ anti-diagonal. The {\em point-to-point rectangle} sub-graphs $P_{N,M}$ (P for {\em point}) consist of all the paths of $L_N$ or $C_N$ (restricted to the pertinent rectangle) with a single vertex on the $N$'th anti-diagonal, the vertex $(N+1-M,M)$.
The {\em point-to-point square} two sub-graphs $P_{N,{{N+1} \over 2}}$ (upward/right or complete) are well defined for $N$ odd.
Figure \ref{twopaths} displays a path in $C_8$ in green and a path in $L_8$ or in $P_{8,5}$ in red.

The subject can be generalized to ${\mathbb Z}^d$, where the role of right-angle triangles and rectangles is played by orthogons (simplices and orthotopes). For simplicity, focus is restricted to ${\mathbb Z}^2$.

\medskip

\begin{figure}[hbt!]
\caption{Oriented $N$-partite graphs. The red path is in $L_8$, the green path is in $C_8$.}
\begin{center}
\begin{tikzpicture} \label{twopaths}
\draw[gray, very thick] (1,1) -- (8,1);
\draw[gray, very thick] (1,1) -- (1,8);
\draw[gray] (2,1) -- (2,7);
\draw[gray] (3,1) -- (3,6);
\draw[gray] (4,1) -- (4,5);
\draw[gray] (5,1) -- (5,4);
\draw[gray] (6,1) -- (6,3);
\draw[gray] (7,1) -- (7,2);
\draw[gray] (1,8) -- (8,1);
\draw[gray] (1,2) -- (7,2);
\draw[gray] (1,3) -- (6,3);
\draw[gray] (1,4) -- (5,4);
\draw[gray] (1,5) -- (4,5);
\draw[gray] (1,6) -- (3,6);
\draw[gray] (1,7) -- (2,7);
\draw[gray] (1,8) -- (8,1);
\draw[red, very thick] (1,1) -- (1,2);
\draw[red, very thick] (1,2) -- (2,2);
\draw[red, very thick] (2,2) -- (2,3);
\draw[red, very thick] (2,3) -- (2,4);
\draw[red, very thick] (2,4) -- (3,4);
\draw[red, very thick] (3,4) -- (4,4);
\draw[red, very thick] (4,4) -- (4,5);
\draw[green, very thick] (1,1) -- (2,1);
\draw[green, very thick] (2,1) -- (1,3);
\draw[green, very thick] (1,3) -- (3,2);
\draw[green, very thick] (3,2) -- (5,1);
\draw[green, very thick] (5,1) -- (3,4);
\draw[green, very thick] (3,4) -- (4,4);
\draw[green, very thick] (4,4) -- (5,4);
\filldraw[black] (0,1) circle (0pt) node[anchor=west]{$(1,1)$};
\filldraw[black] (5,4) circle (0pt) node[anchor=west]{$(5,4)$};
\filldraw[black] (4,5) circle (0pt) node[anchor=west]{$(4,5)$};
\end{tikzpicture}
\end{center}
\end{figure}

\medskip

\noindent {\bf Last passage percolation in ${\mathbb Z}^2$}. Each vertex of $C_N$ is assigned a non-negative random variable. These {\em weights} $W(i,j)$ (canonically, $W$) are identically $F$-distributed "times", and the commonly studied {\em i.i.d. model} makes them independent and identically distributed. The {\em convex bound model} makes them identically distributed but dependent, in a manner and purpose to be described.

For each path $\eta$ in $C_N$, let $\sigma_\eta$ be the sum of the $N$ weights $W(i,j)$ along $\eta$. For each of the graphs $C_N$, $L_N$ and $P_{N,M}$, the {\em last passage percolation} (LPP) time
is the maximum of the sums $\sigma_\eta$ over all paths $\eta$ on the corresponding graph, to be denoted by $R_N$ for point-to-line percolation on $L_N$ and $S_{N,M}$ for point-to-point percolation on $P_{N,M}$. LPP times on $C_N$ will be appended the letter $C$ in the subscript, and the convex bound coupling will be indicated by a superscript star ($*$).

\medskip

\noindent {\bf Stochastic and convex dominance} The distribution $F_V$ on ${\cal R}$ of a r.v. $V$ dominates another $F_W$ of $W$ stochastically if $F_V(x) \le F_W(x)$ pointwise. Equivalently, if $E[\Psi(V)] \ge E[\Psi(W)]$ for all non-decreasing $\Psi$ for which these expectations exist. Under a weaker type of dominance, $F_V$ dominates $F_W$ (both in ${\cal L}_1)$ in the increasing convex order (Shaked and Shanthikumar \cite{Shaked}) if $E[(V-x)^+]=\int_x^\infty (1-F_V(t))dt \ge \int_x^\infty (1-F_W(t))dt=E[(W-x)^+]$ pointwise. Equivalently, if $E[\Psi(V)] \ge E[\Psi(W)]$ for all non-decreasing, convex $\Psi$ for which these expectations exist. Henceforth, {\em increasing} will be tacit.

\medskip

\noindent {\bf Convexly maximal LPP times}. With this terminology, (worst case) couplings will be found such that their LPP times convexly dominate those of any other coupling, e.g. the i.i.d. model. In particular, the maximal values $E[R_N^*], E[S_{N,M}^*], E[R_{C,N}^*], E[S_{C,N,M}^*]$ will be identified. While the LPP times on $C_N$ and $L_N$ for the i.i.d. model grow with $N$ at different rates, it will be seen that $R_N^*$ and $R_{C,N}^*$ are identically distributed, as are $S_{N,M}^*$ and $S_{C,N,M}^*$.

Although LPP on $C_N$ has not been considered of interest, it is presented because its convex bounds are easier to identify and provide a benchmark against which to compare the LPP times on the graphs of interest: Since $R_{C,N} \ge R_N$ a.s., if the distribution of $R_{C,N}^*$ is identified and there is a coupling on $L_N$ with $R_N$ distributed like $R_{C,N}^*$, then $R_{C,N}^*$ is the convexly maximal point-to-line LPP time as well.

The distributions of the various convex bound LPP times will be presented in nearly closed form for each $F$, $N$ and $(N,M)$, and in actual closed form for exponential, uniform and Pareto $F$.

In particular, it will be proved for point-to-line LPP with $Exp(1)$-distributed weights, that ${R_N^* \over N}$ and ${R_{C,N}^* \over N}$ are distributed like the shifted exponential variable $W+{{\log(N!)} \over N}=W - 1 + \log(N) + o(1)$. The point-to-point counterparts (with $M \le {{N+1} \over 2}$) are distributed as the shifted exponential variables $W+2 {\log((M-1)!) \over N} + (1-2 {{{M-1}} \over N})\log(M)=W + \log(N) + \log(\gamma) - 2 \gamma$ + o(1). This asymptotic unimodal symmetric shape function $\log(\min(\gamma,1-\gamma)) - 2 \min(\gamma,1-\gamma)$ (with minimal value $-\infty$ and maximal value $-1-\log(2)$) differs from the asymptotic unimodal symmetric shape function $(\sqrt{\gamma}+\sqrt{1-\gamma})^2$ of the i.i.d. model by R\"ost \cite{Rost} (with maximal value twice its minimal value).
Unlike the i.i.d. model for $L_N$ for which the normalized LPP time is asymptotically constant, the convex bound normalized LPP time stays random and grows to $\infty$ like $\log(N)$. The i.i.d. model for $C_N$ has normalized LPP with mean growing like $\log(N)$ as well.

It may be of interest to notice that the point-to-line mean normalized convex bound exceeds asymptotically by $\log(2)$ the point-to-point central (and maximal) value, while the two are equal for the i.i.d. case.

\medskip

\noindent {\bf Remark on physical viability}. The weight couplings presented in this study do not make "real-world" sense and are not offered as an alternative physically viable model. They are a means to identify sharp upper bounds on expected LPP times if weights are identically distributed but not necessarily independent. The same applies to the brief treatment in Section \ref{convexminimum} of couplings achieving lower bounds, via completely mixable distributions. These bounds are not approximations to the i.i.d. case. Rather, they provide a warning that LPP can be very different if the independence assumption is relaxed.

\medskip

\noindent{\bf Organization}. Section \ref{convmax} describes various representations of probability distributions and stochastic dominance relations, with emphasis on LPP-type maxima of partial sums.
Section \ref{completeLPP} presents the convex upper bound analysis of LPP on the complete $N$-partite graph $C_N$ and illustrates the analysis on weight distributions that are "memoryless up to scaling". Section \ref{flatRW} introduces a notion of "flat" random walk as a general random environment that constructs couplings whose LPP times have the same distribution as LPP on $C_N$. Subsection \ref{flatisunif} shows that LPP on $L_N$ admits a flat random walk representation, via a P\'{o}lya-urn-scheme construction. Somewhat off this organizational order, Section \ref{convexminimum} complements the rest of the paper by analyzing lower bounds to LPP.

\section{Convex majorization} \label{convmax}

\subsection{Characterizations of a distribution} \label{charactdist}

Attention will be restricted to continuous cumulative distribution functions (cdf) $F$ with finite mean, strictly increasing from the essential infimum $0$ to the essential supremum, finite or $\infty$.
The cdf $F(x)=P(T \le x)$ of the r.v. $T$ leads to the survival function $F^*(x)=1-F(x)$, premium function $H_F(x)=E[(T-x)^+]= \int_x^\infty F^*(t)dt$, mean residual function $g_F(x)= E[T-x|T>x]={{H_F(x)} \over {F^*(x)}}$ and upper-quantile function
\begin{equation}  \label{upperquanyt}
Q_F(u)=(F^*)^{-1}(u)=\inf \{x | F(x)>1-u\} \ , \ 0 < u < 1 \ .
\end{equation}
The essential infimum $Q_F(1)=\inf \{x | F(x)>0\}$ and essential supremum $Q_F(0)=\sup \{x | F(x)<1\}$ are limiting values. If $U \sim U(0,1)$, then $Q_F(U) \sim F$. Each of the functions $F, F^*, H_F, g_F$ and $Q_F$ fully characterizes the distribution of $T$. For $v \in (Q_F(1), Q_F(0))$, let $F_L^{(v)}(t)= \min(1,{{F(t)} \over {F(v)}})$ stand for the cdf of the conditional distribution of $T$ given that $T<v$ and $F_R^{(v)}(t)=\max(0,{{F(t)-F(v)} \over {1-F(v)}})$ for the cdf of the conditional distribution of $T$ given that $T>v$. For $n \ge 1$, $F_L^{(n)}(t)$ and $F_R^{(n)}(t)$ will stand for $F_L^{(w_n)}(t)$ and $F_R^{(w_n)}(t)$, where $w_n=Q_F({1 \over n})$ is the ${1 \over n}$-upper-quantile of $F$. In simpler terms, $1-F_R^{(n)}(t)=\min(1,n(1-F(t)))$.

\subsection{The concepts of stochastic and convex majorization} \label{conceptsmajor}

$V \sim F_V$ is {\em stochastically bigger} than $W \sim F_W$ if $E[\Psi(V)] \ge E[\Psi(W)]$ for all $\Psi$ non-decreasing for which these expectations exist. Equivalently, if $P(V > x) = E[I_{(x,\infty)}(V)] \ge E[I_{(x,\infty)}(W)] = P(W > x)$ for all $x$. Equivalently, if in some probability space there are r.v. $(V,W)$ with the given marginals such that $V \ge W$ a.s.

For integrable r.v.'s, $V \sim F_V$ is {\em convexly bigger} than $W \sim F_W$ if $E[\Psi(V)] \ge E[\Psi(W)]$ for all $\Psi$ non-decreasing and convex for which these expectations exist. Equivalently, if for all $x$, $E[(V-x)^+] = \int_x^\infty P(V > t) dt \ge \int_x^\infty P(W > t) dt = E[(W-x)^+]$. Equivalently, if in some probability space there is a submartingale pair $(W,V)$ with the given marginals. I.e., such that $E[V|W] \ge W$ a.s.

Stochastic and convex inequality are relations between the distributions of r.v.

\subsection{Stochastic majorization of the maximum of r.v.'s with given marginals: maximally dependent couplings and the Fr\'{e}chet bound} \label{Frechet}

Let $X_i \sim F_i \ , 1 \le i \le n$. A {\em coupling} is a joint distribution with the prescribed marginal distributions. For every coupling, $P(\max_{i=1}^n X_i > x)=P(\bigcup_{i=1}^n \{X_i>x\})) \le \sum_{i=1}^n F_i^*(x)$. Hence (the Fr\'{e}chet bound),
\begin{equation} \label{Frechetbound}
P(\max_{1 \le i \le n} X_i > x) \le \min(1,\sum_{i=1}^n F_i^*(x)) = H^*(x) \ .
\end{equation}

Let $x_0=\inf \{x | \sum_{i=1}^n F_i^*(x) <1\}$. Then $\sum_{i=1}^n F_i^*(x_0) = 1$. The following class of joint distributions achieve equality in inequality (\ref{Frechetbound}): Sample once from the set $\{1,2,\dots,n\}$ with respective probabilities $(F_1^*(x_0),F_2^*(x_0),\dots,F_n^*(x_0))$. For outcome $i$, sample $X_i^*$ from the conditional \linebreak $F_i$-distribution on $[x_0, \infty)$ and sample each other $X_j^*$ from the conditional $F_j$-distribution on $(-\infty, x_0)$. It should be clear that for each $i$, $X_i^* \sim F_i$, so these joint distributions, termed {\em maximally dependent} in the literature, have the given marginals. For each $x \ge x_0$, the events $\{X_i^*>x\}$ are mutually disjoint, turning inequality (\ref{Frechetbound}) into an equality.

Summarizing, for each coupling, $P(\max_{i=1}^n X_i>x) \le H^*(x) = P(\max_{i=1}^n X_i^*>x)$. In particular, if the $X_i$ are $F$-distributed, then $x_0=(F^*)^{-1}({1 \over n})$ and the distribution function $H$ has essential infimum $x_0$ and survival function $H^*(x)=nF^*(x)$ for $x \ge x_0$. Letting $U \sim U(0,1)$, $Y_n=(F^*)^{-1}({U \over n}) \sim H$.

\medskip

This describes how is the maximal weight along an anti-diagonal made stochastically maximal.
As an example, if the $X_i$ are identically $Exp(1)$ distributed, then $x_0=\log(n)$ and $H$ is the distribution of the shifted exponential variable $X_1 + \log(n)$, with mean $1+\log(n)$. For the independent coupling, the mean is (in terms of Euler's constant) $\sum_{i=1}^n {1 \over i} \approx 0.5772 + \log(n)$.

\subsection{Convex majorization of the sum of r.v.'s with given marginals: \linebreak Co-monotone coupling} \label{sumssub}

Let $Y_i \sim H_i \ , 1 \le i \le n$ have finite means. By means of one observation $U \sim U(0,1)$, define $Y_i^*=(H_i^*)^{-1}(U)$ for each $i$. This is the co-monotone coupling. Whatever be the joint distribution of the $Y_i$ with the given marginals, $\sum_{i=1}^n Y_i^*$ is convexly bigger than $\sum_{i=1}^n Y_i$.

A proof is contained within the proof of Theorem 2 in Meilijson \& N\'{a}das \cite{MeilijsonNadas}.

If the $H_i$ are all equal, all $Y_i^*$ are identical a.s. and $n Y_1$ convexly dominates $\sum Y_i$ for every coupling. Since $\sum_{i=1}^n E[Y_i]$ is the same for all couplings, a convexly larger sum is {\em more dispersed} rather than {\em bigger}.

\subsection{Convex majorization of maxima of partial sums - The method of Meilijson \& N\'{a}das \cite{MeilijsonNadas}} \label{maxsub}

This subsection generalizes the previous two, and places LPP-type convex majorization in a broader context. It provides a (tacit) direct proof that $R_N^*$ convexly maximizes $R_N$, without the fortuitous Deus Ex Machina solution via LPP on the complete $N$-partite graph.

\medskip

$X_n \sim F_n$, for $1 \le n \le N$, are non-negative r.v. with some joint distribution with the given marginal distributions. $I_k$, for $1 \le k \le K$, are non-empty subsets of $\{1,2,\dots,N\}$. $S_k=\sum_{n \in I_k} X_n$ are partial sums, and ${\cal M}=\max_{k=1}^K S_k$ is the maximal partial sum, such as LPP. Subsection \ref{Frechet} (the "parallel" machine) dealt with singleton subsets, and Subsection \ref{sumssub} (the "series" machine) dealt with a single subset, the set itself. LPP on $C_N$ is a machine built by nesting series and parallel sub-machines.

A convex bound for the distribution of ${\cal M}$ will be found by maximizing over the joint distributions the premium function $E[({\cal M}-x)^+]$, for each $x$. Take an arbitrary joint distribution, fix an arbitrary vector $(v_1,v_2,\dots,v_N)$ and express for each $k$,
$$
S_k-x= \sum_{n \in I_k} (X_n-v_n) + \sum_{n \in I_k} v_n - x \ .
$$
Increase the RHS as indicated to obtain
$$
S_k-x \le \sum_{n \in I_k} (X_n-v_n)^+ + (\max_k \sum_{n \in I_k} v_n -x)^+ \ .
$$
Increase further the RHS as indicated to obtain
$$
S_k-x \le \sum_{n=1}^N (X_n-v_n)^+ + (\max_k \sum_{n \in I_k} v_n -x)^+ \ .
$$
The RHS in constant wrt k, so
$$
{\cal M}-x = \max_k S_k-x \le \sum_{n=1}^N (X_n-v_n)^+ + (\max_k \sum_{n \in I_k} v_n -x)^+ \ .
$$
The RHS is non-negative, so
\begin{equation} \label{pointwise}
({\cal M}-x)^+ = (\max_k S_k-x)^+ \le \sum_{n=1}^N (X_n-v_n)^+ + (\max_k \sum_{n \in I_k} v_n -x)^+ \ .
\end{equation}
This inequality holds a.s., so, taking expectations,
\begin{equation} \label{Bfunction1}
H_{\cal M}(x)=E[({\cal M}-x)^+] \le \sum_{n=1}^N H_{X_n}(v_n) + (\max_k \sum_{n \in I_k} v_n -x)^+ \ .
\end{equation}
The LHS is constant in $v$, leading to
\begin{equation} \label{Bfunction2}
H_{\cal M}(x)=E[({\cal M}-x)^+] \le B(x)=\inf_v \{\sum_{n=1}^N H_{X_n}(v_n) + (\max_k \sum_{n \in I_k} v_n -x)^+\} \ .
\end{equation}

It should be observed that the function $B$ in (\ref{Bfunction2}) depends only on the marginal distributions $F_n$, and it bounds from above $H_{\cal M}$ for every joint distribution of the $X_n$ with the given marginals.

But (\ref{Bfunction2}) is not quite the end of the story. The function $B$ is non-negative, convex and decreasing. If $B'(x)$ is ever less than $-1$, there is a value $x_0$ of $x$ where $B'(x_0)=-1$. Keep $B$ as is on $[x_0,\infty)$ and replace $B$ on $[0,x_0]$ by the linear function with slope $-1$ with the given value at $x_0$. Then this modified $B$ is the premium function of a distribution, intended to be the distribution of ${\cal M}^*$. Its essential infimum is $x_0$. For $x \ge x_0$, optimal $v$ should make $(\max_k \sum_{n \in I_k} v_n -x)^+ =0$ and the unconstrained problem may be solved by the minimization of the separable convex function $\sum_{n=1}^N H_{X_n}(v_n)$ under K linear constraints $\sum_{n \in I_k} v_n \le x$.

The analysis in \cite{MeilijsonNadas} proceeds to show that $B(x)$ is indeed the maximal value of $H_{\cal M}(x)$, by building a joint distribution for which all inequalities in (\ref{pointwise}) hold as equalities. This analysis is done by determining the solution vector $v$ of (\ref{Bfunction2}) and the Lagrange multipliers corresponding to the linear constraints. The Lagrange multiplier for a subset is the probability that this subset is the geodesic, under the worst-case coupling.

\medskip

Couplings achieving equality in (\ref{Bfunction2}) may depend on the particular value of $x$. Hence, $B$ is a tight convex upper bound that need not correspond to a member of the family. It was proved in \cite{MeilijsonNadas} that if the graph is built by nesting series and parallel components, there is a coupling that works for all. As mentioned above, the $N$-partite graph consists of parallel machines (maximal weights in the anti diagonals) connected in series (sums of these anti-diagonal maxima).

\section{LPP on the complete $N$-partite graph} \label{completeLPP}

Convex majorization on the complete $N$-partite graph is brought first, as it is easy to solve, for point-to-line LPP as well as rectangle point-to-point LPP. Furthermore, as the remainder of this study will show (Theorem \ref{theoremflat}), {\bf it provides the tight answers for the various upward/right sub-graphs too}.

\medskip

\noindent {\bf A diversion to the i.i.d. model}. Asymptotic analysis of the i.i.d. model LPP on the complete $N$-partite graph could be performed by means of Gnedenko \cite{Gnedenko} maximal laws. Gnedenko proved that (linearly normalized) maxima of i.i.d. r.v. can only have three limit laws, the Fr\'{e}chet, Gumbel and Reversed Weibull distributions, and characterized the corresponding three domains of attraction. Thus, the Gnedenko laws model asymptotics for the maxima along the anti-diagonals, providing ingredients for a Central Limit approach to the sum of these independent maxima over the anti-diagonals. The next subsection deals with a representative from each of the three classes, but restricts analysis to the subject matter of this study, the development of convex bounds. The i.i.d. model is introduced sporadically.

\medskip

Unlike Subsection \ref{Frechet} where LPP weights $W(i,j)$ along an anti-diagonal are identically \linebreak $F$-distributed, the relevant application of the material in Subsection \ref{sumssub} to LPP involves non-identical distributions $H_i$. It should now be clear that the convexly maximal LPP time for the $N$-partite graph can be characterized in terms of a single uniform observation $U \sim U(0,1)$ as $\sum_{j=1}^N (F^*)^{-1}({U \over n_j})$, where the sum runs through the $N$ anti-diagonals and $n_j$ is the length of the $j$'th anti-diagonal. The single uniform observation $U$ is the raw material for the maxima in the anti-diagonals, and makes these maxima co-monotone over the anti-diagonals. Thus,

\begin{theorem} \label{theoremcomplete}
{\bf The convex bound model for the complete $N$-partite graph}. In terms of a single observation $U =F^*(W(1,1)) \sim U(0,1)$, the point-to-line LPP time (see (\ref{upperquanyt}))
\begin{equation} \label{cbmodelptl}
R_{C,N}^*=\sum_{j=1}^N (F^*)^{-1}({U \over j}) = \sum_{j=1}^N Q_F({U \over j})
\end{equation}
and the point-to-point LPP time (for $M \le {{N+1} \over 2}$)
\begin{equation} \label{cbmodelptp}
S_{C,N,M}^*=2 \sum_{j=1}^M Q_F({U \over j})+ (N-2M)Q_F({U \over M})
\end{equation}
corresponding to the couplings with maximally dependent weights on anti-diagonals and sum of co-monotone maxima over the anti-diagonals, convexly dominate their counterparts for any other coupling.
\end{theorem}

The weight $W(1,1)$ at the origin determines
$U=F^*(W(1,1))$, which in turn determines the worst-case LPP times.

Since weights are non-negative, the power functions are increasing and convex. Hence, each moment of $R_{C,N}^*$ and $S_{C,N,M}^*$ is maximal in the class of all weight couplings.
Among all couplings with $E[R_{C,N}]=E[R_{C,N}^*]$, $R_{C,N}^*$ has maximal variance. There is a wide spectrum of such maximal-expectation couplings. Plainly, instead of using in (\ref{cbmodelptl}) and (\ref{cbmodelptp}) the same $U$ {\em between} anti-diagonals, apply any copula. As will be shown in Section \ref{convexminimum}, the minimal possible variance of these maximal-mean LPP times for exponential weights is asymptotically zero.

\medskip

Theorem \ref{theoremcomplete} will be illustrated on a menu of distributions, those that are memoryless up to scaling.

\subsection{Distributions that are memoryless up to scaling} \label{uptoscaling}

The distribution $F$ is {\em memoryless up to scaling} if $P(T>x+g(x)s|T>x)=P(T>s)$ for some function $g$. In other words,
\begin{equation} \label{memoscaling}
{{F^*(x+g(x)s)} \over {F^*(x)}}= F^*(s) \ .
\end{equation}
Equivalently, $F^*(x+g(x)s) = F^*(x) F^*(s) = F^*(s+g(t)x))$. Hence, $x+g(x)s=s+g(s)x$ or $g(x)=1+c x$. Identity (\ref{memoscaling}) takes the symmetric form $F^*(x + s + c x s)=F^*(x) F^*(s)$.

\medskip

\noindent {\bf The $\beta$ distribution, in Gnedenko's Reversed Weibull class}. For $c<0$, $F$ has support $(0,-{1 \over c}]$ and $F^*(x)=(1-|c| x)^\beta$, for arbitrary $\beta>0$, meets the memoryless condition, perhaps uniquely. This corresponds to the $\beta$ distribution with the first parameter equal to $1$. The uniform distribution is the special case where $\beta=1$.

\medskip

\noindent {\bf The exponential distribution, in Gnedenko's Gumbel class}. For $c=0$, $F$ is exponential, with arbitrary parameter $\theta>0$. This is the limit as $k \rightarrow \infty$ of the case $c<0$, letting $c=-{1 \over k}$ and $\beta = \theta k$.

\medskip

\noindent {\bf The Pareto power-law distribution, in Gnedenko's Fr\'{e}chet class}. For $c>0$, $F$ has support $(0,\infty)$, and $F^*(x)=(1+c x)^{-\alpha}$ satidfies the memoryless condition for arbitrary $\alpha>1$, so as to have finite expectation. Needless to say, the exponential case is the limit of this case as well, letting $c={1 \over k}$ and $\alpha= k \theta$. This distribution has essential infimum $0$. The name Pareto distribution is commonly given to its shift by $1$.

\begin{theorem} \label{LPPmemoryless}
For $W \sim F$ that is memoryless up to scaling, the convex bound LPP times for the complete $N$-partite graph are distributed like linear functions of $W$.
\end{theorem}

\noindent {\bf Proof}. Identify $U$ as $U=F^*(W(1,1))$ and plug in $x=(F^*)^{-1}({1 \over j})$ and $s=W(1,1)$ in (\ref{memoscaling}). This leads to the a.s. identity (see (\ref{upperquanyt}))
\begin{equation} \label{asidentity}
(F^*)^{-1}({U \over j})=(F^*)^{-1}({1 \over j}) + g((F^*)^{-1}({1 \over j})) W(1,1) = W(1,1) + Q_F({1 \over j})(1 + c W(1,1)) \ ,
\end{equation}
which shows that every summand in (\ref{cbmodelptl}) and (\ref{cbmodelptp}) is a linear function of $W(1,1)$ (exceeding $W(1,1)$). Hence, so is each of the two types of sum.

\medskip

The coefficients of these linear functions will be derived for each case separately.

\subsection{Illustration of convexly dominant LPP times} \label{illustratememoryless}

\noindent {\bf The exponential distribution}. $W \sim F$, with $f(t)=F^*(t)=H_F(t)=\exp\{-t\}$ and $g_F(t)=1$ ($c=0$). The terms to add in (\ref{cbmodelptl}) are easily identifiable in (\ref{asidentity}) as $W(1,1)+\log(j)$. Hence,
\begin{equation} \label{completeexpoptl}
R_{C,N}^* = N W(1,1) + \log(N!) \ .
\end{equation}
Similarly, the sum in (\ref{cbmodelptp}) is $2 M W(1,1)+2\log(M!)+(N-2M)(\log(M) + W(1,1))$. I.e., for $M \le {{N+1} \over 2}$,
\begin{equation} \label{completeexpoptp}
S_{C,N,M}^* = N W(1,1) + 2\log(M!)+(N-2M)\log(M) \ .
\end{equation}

For $M=1$, $S_{C,N,1}^* = N W(1,1)$ checks as it should be. For the central point-to-point value $M={N \over 2}$, $S_{C,N,{N \over 2}}^* = N W(1,1) + 2\log({N \over 2}!)$ falls behind $R_{C,N}^*$ asymptotically by $(N+1) \log(2) - {1 \over 2} \log(2 \pi N)$. Thus, the normalized LPP time, of order $\log(N)$, falls behind by $\log(2)$, as announced in Section \ref{techprelim}, just before the remark on physical viability.

The analysis of $R_{C,N}$ for the i.i.d. exponential case is made amenable by the well known fact that the maximum of $n$ i.i.d. $Exp(1)$ r.v. may be expressed as the sum of $n$ independent exponentially distributed variables, with parameters from $1$ to $n$. This fact identifies the mean of the maximum on the n'th anti-diagonal as the harmonic sum $J_n=\sum_{i=1}^n {1 \over i}$, its asymptotic variance as ${\pi^2 \over 6}$ and the asymptotic distribution of $R_{C,N}$ as Gaussian.
Thus, both $E[R_{C,N}]$ for the independent coupling and $E[R_{C,N}^*]$ are $N\log(N)$ plus smaller order terms: $E[R_{C,N}^*]$ is $N\log(N)$ plus a term of order $\log(N)$, while $E[R_{C,N}]$ for the independent coupling is $N\log(N)$ minus a term of order $N$. The variance of the normalized LPP time ${{R_{C,N}} \over {N}}$ goes to zero and the variance of the normalized LPP time ${{R_{C,N}^*} \over {N}}$ is $1$, independently of $N$.

\medskip

\noindent {\bf The uniform distribution}. $W \sim F$, with $f(t)=1, F^*(t)=1-t , H_F(t)={1 \over 2}(1-t)^2$ and $g_F(t)={{1-t} \over 2}$, on $(0,1)$. The terms to add in (\ref{cbmodelptl}) are $(F^*)^{-1}({U \over j})=1-{U \over j}$. Hence, (recalling \linebreak $J_n=\sum_{i=1}^n {1 \over i} (\approx \log(n)+0.5772)$ and) interpreting $U=1-W(1,1)$,
\begin{equation} \label{completeunifptl}
R_{C,N}^* = J_N W(1,1) + N - J_N
\end{equation}
is uniformly distributed in the interval $(N-J_N,N)$.

For the i.i.d. uniform case, the LPP time $R_{C,N}$,  with mean $\sum_{n=1}^N {n \over {n+1}}= N+1-J_{N+1}$, is asymptotically Gaussian with variance $2-{\pi^2 \over 6}=0.3551$. Details are omitted.

The sum in (\ref{cbmodelptp}) is $2(M- U J_N) + (N - 2 M)(1-{U \over M})$. I.e., for $M \le {{N+1} \over 2}$,
\begin{equation} \label{completeexpoptp}
S_{C,N,M}^* = ({N \over M}+2(J_M-1)) W(1,1) + N - ({N \over M}+2(J_M-1)) \ .
\end{equation}

For $M=1$, $S_{C,N,M}^*$ is indeed $N W(1,1)$. More generally, it is uniformly distributed in the interval $(N - ({N \over M}+2(J_M-1)),N)$. For $M \approx {N \over 2}$, the interval is $[N - 2 J_{N \over 2}, N]$, about twice as long as for point-to-line LPP. Re point-to-line LPP,
$E[R_{C,N}^*]=N-{J_N \over 2}$ can be compared with the independent coupling, yielding $E[R_{C,N}]=\sum_{i=1}^N {i \over {i+1}}=N-J_N+{N \over {N+1}}$, approximately ${{\log(N)} \over 2}$ smaller than $E[R_{C,N}^*]$, all three $N-O(\log(N))$.

\medskip

\noindent {\bf The Pareto distribution with tail parameter $\alpha>1$}. This heavy-tail distribution is supported by $(1,\infty)$, with survival function $F^*(t)= t^{-\alpha}$, whose inverse is $(F^*)^{-1}(u)=u^{-{1 \over \alpha}}$. The summands of $R_{C,N}^*$, as defined in (\ref{asidentity}), are $(F^*)^{-1}({U \over j})=W(1,1) j^{1 \over \alpha}$.
Let $K_n=\sum_{j=1}^n j^{1 \over \alpha} (\approx {\alpha \over {\alpha+1}}n^{1+{1 \over \alpha}})$.
For this family,
\begin{equation} \label{completeparetoptl}
R_{C,N}^* = K_N W(1,1).
\end{equation}
and (for $M \le {{N+1} \over 2}$)
\begin{equation} \label{completeparetoptp}
S_{C,N,M}^* = (2 K_M + (N - 2 M) M^{1 \over \alpha}) W(1,1) \ .
\end{equation}

Again, the case $M=1$ yields $N W(1,1)$. The maximal and central point-to-point value at $M={N \over 2}$ yields ${{S_{C,N,{N \over 2}}^*} \over {R_{C,N}^*}} \approx 2^{-{1 \over \alpha}}$.

\section{On small last passage times} \label{convexminimum}

This is a detour from the main theme studying how large can LPP time be, to the opposite question, how small can LPP times be. As will be seen, sharp answers can be provided for distributions that are memoryless up to scaling.

Since LPP time exceeds the total weight along any path, expected LPP time is at least $N E[W]$. By making all $W(i,j)$ along every anti-diagonal equal to each other, LPP time has mean $N E[W]$, regardless of the joint distribution of these anti-diagonal weights. In this section, the weights will be denoted by $W_1, W_2, \dots, W_N$, where $W_j$ is the identical value of the weights $W(\cdot, \cdot)$ on the $j$'th anti-diagonal.

Thus, the minimal expected LPP time in the class of all couplings of the weights $W(i,j)$ is simply $N E[W]$.
In the exponential case, this means that the i.i.d. LPP times have expected value at most twice the minimal possible value.

As for the variance of these LPP times with common anti-diagonal weights, there are distributions $F$, the {\em completely mixable} distributions (see Wang \& Wang \cite{Wang} and its references, R\"{u}schendorf \& Uckelman \cite{Ruschendorf} in particular) for which there is a coupling of $W_1, W_2, \dots W_N$, with sum a.s. constant. So, in this case, the minimal possible variance of these LPP times is zero (even if $F$ has finite mean and infinite variance). The uniform distribution and any distribution with
symmetric unimodal density are completely mixable. As an example,

\medskip

\noindent {\bf The exchangeable multivariate Gaussian case}. Let $Z_i, 1 \le i \le N$ be i.i.d. standard normal, and let $Q_N$ be their sum. Let $W_i = {{N Z_i-Q_N} \over {\sqrt{N (N-1)}}}, 1 \le i \le N$. These $W_i$ are standard normal and add up to $0$ a.s. A Gaussian copula that generates this singular multivariate Gaussian distribution is $\Phi(W_i), 1 \le i \le N$. This solution is not unique: giving up on exchangeability and correlation coefficient $\rho=-{1 \over {N-1}}$ between every pair, for $N=2 k$ even, letting $W_i$ be i.i.d. for ${i \le k}$ and letting $W_{k+i}=-W_i$ for each ${i \le k}$, will also achieve ${\mbox Var}[\sum_{i-1}^{N} W_i]=0$. For $N$ odd, apply the exchangeable trivariate Gaussian solution once and complement it with bivariate pairs.

The role played by symmetry is apparent. Minimal variance of the sum is zero for $F$ symmetric and $N$ even, by the pair-wise construction above. The case $N=3$ (and then all $N \ge 2$) for symmetric unimodal densities falls under the rubric of completely mixable distributions of Wang \& Wang \cite{Wang}.

\medskip

\noindent {\bf The $Exp(1)$ case with $N=2$}. Clearly, the minimal variance is achieved by the antithetic pair $W_1=-log(U)$ and $W_2=-log(1-U)$ with sum $W_1+W_2=-log(U(1-U))$. By numerical integration and large sample simulation, ${\mbox Var}[W_1+W_2] = 0.7101$, smaller than ${\mbox Var}[W_1] = 1$. In fact, for every $F$ with finite variance, a coupling with minimal variance of $W_1+W_2$ is $W_1=F^{-1}(U), W_2=F^{-1}(1-U)$.

\medskip

A significant difference between $N=2$ and $N>2$ is that there is a notion of antithetic copula for $N=2$ but not for any $N>2$. In contrast, the co-monotone copula with identical uniform variables and the independent copula with i.i.d. uniform variables, naturally exist for all $N$. For any distribution $F$ with finite variance, just as in the exponential case above, the minimal possible variance of the sum of two $F$-distributed r.v. is attained by the sum $F^{-1}(U)+F^{-1}(1-U)$, with $U \sim U(0,1)$, generated by the antithetic copula.

\subsection{Completely mixable distributions with bounded support.
Wang \& Wang \cite{Wang}}

For every distribution $F$ with decreasing density, supported by the finite interval $[0,b]$, with mean $\mu \in [{b \over N}, b-{b \over N}]$, there is a coupling of $N$ $F$-distributed r.v. $W_i$ such that $\sum_{i=1}^N W_i = N \mu$ a,s.

\medskip

This theorem by Wang \& Wang provides a tool to analyze the variance of the sum of $N$ identically distributed weights with distribution memoryless up to scaling. The $\beta$ distributions in Subsection \ref{uptoscaling} have bounded support and either constant or decreasing density. Hence, by the theorem of Wang \& Wang, for every $N \ge 2$ there is a coupling with constant LPP time $N E[W]$.

\medskip

The exponential and Pareto distributions, supported by right rays $[x_0, \infty)$, have decreasing density.

The plan to bound from above the minimal variance of the sum $S_N$ of $N$ weights $W_i$ that are either exponentially or Pareto distributed is as follows. Fix $b>0$. Flip a coin with $P(H)=P(W>b)$. If it falls on $H$ ($I_H=1)$, let $W_i=b+Y_i$ (exponential case) or $W_i=b Y_i$ (Pareto case), where the $Y_i$ are distributed according to the same coupling we are searching for globally. If the coin falls on $T$ ($I_H=0$), $(W_1, W_2, \dots, W_N)$ is sampled via a copula as in Wang \& Wang \cite{Wang}, with deterministic sum, from the conditional distribution of $W$ given that $W<b$. The choice $N = {b \over {E[W | W<b]}}$ works well. This procedure defines a family of distributions, parameterized by $b$, with well defined $N$ and well defined $V_N={\mbox Var} \sum_{i=1}^N W_i$ as functions of $b$. It will be proved that $V_N$ is asymptotic to $N^4 \exp\{-N\}$ in the exponential case and asymptotically proportional to $N^{4-\alpha}$ in the Pareto case, as $N \rightarrow \infty$. As a result,

\begin{theorem} \label{variancememoryless}
{\bf Variance of small LPP time for distributions that are memoryless up to scaling}. The minimal possible variance of LPP time with minimal possible mean tends to zero as $N \rightarrow \infty$ in the exponential case and in the Pareto case with finite fourth moment. In the uniform case the variance is zero for all $N>1$.
\end{theorem}

\noindent {\bf Proof}: The variance $V_N$ of $S_N$ under the type of coupling described before the statement of the Theorem can be calculated as
\begin{eqnarray}
V_N & = & {\mbox Var}[S_N] = E[{\mbox Var}[S_N | I_H]] + {\mbox Var}[E[S_N | I_H]] \nonumber \\
& = & 0 + P(W>b) V_N + N^2 P(W>b)P(W \le b)(E[W|W>b]-E[W|W \le b])^2 \nonumber \\
& = & N^2 P(W > b)(E[W|W>b]-E[W|W \le b])^2 \ . \label{Varbetweenwithin}
\end{eqnarray}

Since for exponential and finite-mean Pareto distributions ${{E[W|W>b]-E[W|W \le b]} \over b}$ and ${N \over b}$ have finite positive limits as $b \rightarrow \infty$, $V_N$ is asymptotically proportional to $N^4 P(W>N)$. The statement then follows.

\medskip

Keeping the focus of this study on the {\em maximization} of expected functions of LPP times, this detour to small LPP times is {\em almost} over and the task in the next section will be to show that the convex bounds developed above for LPP times in the complete $N$-partite graph apply verbatim to the commonly studied upward/right LPP times on $L_N$ and $P_{N.M}$. Before leaving the completely mixable distributions aside, the next subsection displays a scenario where the small variances achieved for minimal-mean LPP times extend to maximal-mean LPP times as well.

\medskip

\subsection{The variance of maximal-mean LPP time for exponential weights}

It has been shown in this section that there exist couplings of $Exp(1)$ weights with minimal mean LPP time $N$ and asymptotically vanishing variance. The maximal possible mean LPP time has been shown in Subsection \ref{illustratememoryless} to be $N + \log(N!)$. It is to be argued now that there exist couplings with maximal mean LPP time and asymptotically vanishing variance.

Build each anti-diagonal as in Theorem \ref{LPPmemoryless} but instead of using co-monotone (i.e., identical) uniform variables $U$ between diagonals, apply the copula developed in the current section, via completely mixable conditional distributions on $(0,x_0)$. The vanishing variance applies to maximal-mean couplings verbatim, since the LPP contribution of each anti-diagonal to maximal-mean couplings is a shift by a constant of the contribution to minimal-mean couplings.

It should be observed that this argument applies explicitly to the exponential scenario, made possible by the equal coefficients of the contributions over the various anti-diagonals.

\section{Flat random walks on $C_N$} \label{flatRW}

A {\em random walk} on $C_N$ or in $P_{N,M}$ is a stochastic process $B_1, B_2,\dots, B_N$ with initial state \linebreak $B_1=(1,1)$, such that for every $n$, $B_n$ is supported by the $n$'th anti-diagonal (or the pertinent section of it). The random walk is {\em flat} if for every $n$, the marginal distribution of $B_n$ is uniform on the $n$'th anti-diagonal (or the pertinent section). Four types of flat random walk will be singled out.

\medskip

\noindent {\bf The natural random walks} on $C_N$ or $P_{N,M}$ have independent coordinates $B_n$. Trajectories are supported by $C_N$ or by $P_{N,M}$.
\medskip

\noindent  {\bf The P\'{o}lya random walk} on $C_N$ is built as follows. $B_N$ is sampled from its uniform distribution, with ordinal value $O_N$ from $1$ to $N$. Consider an urn with $O_N-1$ "up" balls and $N-O_N$ "right" balls. Balls are successively drawn from the urn without replacement, until the urn is empty. The balls withdrawn describe, backwards from the diagonal to the origin, the steps of the random walk. Its trajectories are supported by $L_N$. The name P\'{o}lya was chosen because this random walk may be built forward, from the origin to the diagonal, by the P\'{o}lya urn scheme. In spite of the elegance of the P\'{o}lya urn scheme forward construction, the backward construction is more naturally generalized to obtain rectangle point-to-point random walks.

\medskip

\noindent {\bf The rectangle random walk} on $P_{N,M}$. Consider the rectangle with corners $(1,1), \linebreak (N-M+1,1), (1,M),(N-M+1,M)$, where ${{N+1} \over 2} \le M \le N$. Two sub-triangles are singled out in this rectangle, $C_{N-M+1}$ with the right angle at $(1,1)$ and its mirror image with the right angle at the opposite corner $(N-M+1,M)$. Choose a random vertex $W_{N-M+1}$ in the sub-diagonal from $(1,N-M+1)$ to $(N-M+1,1)$ and connect it vertically via "up" transitions to the sub-diagonal from $(1,M)$ to $(N-M+1,2 M-N)$. Complete the random walk by two independent samplings without replacement as in the P\'{o}lya random walk construction, to connect to the corners $(1,1)$ and $(N-M+1,M)$. The trajectories of this random walk are supported by $P_{N,M}$.

\medskip

\subsection{A family of couplings generated by flat random walks} \label{flatisunif}

Consider a flat random walk on any of the pertinent LPP graphs. Define
dependent weights $W(i,j)$ as follows, by means of the path chosen by the random walk and one realization $U \sim U(0,1)$, independent of the random walk. For $(i,j)$ along the path chosen by the random walk, let $W(i,j)$, distributed $F_R^{(i+j-1)}(\cdot)$, be
\begin{equation} \label{QuEfAr}
Q_{F_R^{(i+j-1)}}(U)=(1-F_R^{(i+j-1)})^{-1}(U)=(1-F)^{-1}({U \over {i+j-1}})
\end{equation}
and for $(i,j)$ not along the path, $W(i,j)$ is $Q_{F_L^{(i+j-1)}}(U)=(F_L^{(i+j-1)})^{-1}(U)$, distributed $F_L^{(i+j-1)}(\cdot)$.

Let $JD_B$ be the joint distribution of these weights for the flat random walk $B$. Let $LPP_B$ be the sum of weights along the path chosen by the random walk $B$. Equivalently, $LPP_B$ is the last passage percolation time on the collection of paths in the support of the distribution of $B$.

\begin{theorem} \label{theoremflat}
For every flat random walk $B$,
\begin{itemize}
    \item $JD_B$ is a $F$-coupling. I.e., each weight $W(i,j)$ has marginal distribution $F$.

    \item $LPP_B$ has the same distribution for all flat random walks $B$, given by Theorem \ref{theoremcomplete} for the natural flat random walks.

\end{itemize}
\end{theorem}

\noindent {\bf Proof}. $JD_B$ are bona-fide $F$-couplings, because flat random walks have uniform marginals: In the $n$'th anti-diagonal, the mixing probabilities of the two conditional distributions are ${1 \over n}$ and $1-{1 \over n}$, as required.
The path chosen by the random walk is the unique geodesic -- the only path with all weights exceeding the upper quantile. The weights along this path are constructed in a standard way, the same as for the corresponding natural random walk, LPP on the complete $N$-partite graph.

\section*{Acknowledgement}

This study was motivated by a lecture by Riddhipratim Basu in SPA 2023 (Lisbon) in which Ofer Busani's work was quoted, and benefitted from subsequent correspondence with Basu and extensive interaction with Busani.
The author gratefully acknowledges funding by the ISRAEL SCIENCE FOUNDATION (grant No. 1898/21).


\begin{thebibliography}{99}

\bibitem{blairstahn} Blair-Stahn, N. D. (2010). First passage percolation and competition models. ArXiv: Probability, May 4, 2010.

\bibitem{Borodinetal} Borodin, A., Ferrari, P. L., Pr\"{a}hofer, aM. \& Sasamoto, T. Fluctuation
properties of the TASEP with periodic initial configuration. J. Stat.
Phys., 129:1055–1080, 2007.

\bibitem{Feller} Feller, W. (1966). {\em An introduction to Probability theory and its applications, 2}. Wiley, New York.

\bibitem{Gnedenko}  Gnedenko, B.V. (1943). Sur la distribution limite du terme maximum d'une serie al\'{e}atoire. {\em Annals of Mathematics}. {\bf 44 (3)}: 423–453.

\bibitem{Johansson} Johansson, K. (2000). Shape fluctuations and random matrices. {\em Comm. Math. Phys.}, {\bf 209.2}, 437--476.


\bibitem{LaiRobbins} Lai, T. L. \& Robbins, H. (1977). A class of dependent random variables and their maxima. {\em Z. Wahrscheinlichkeitstheorie verw. Gebiete} {\bf 42}, 89--111.

\bibitem{Meilijson} Meilijson, I. (1972). Limiting properties of the mean residual lifetime function. {\em Ann. Math. Statist.}, {\bf 43(1)}, 354--357.

\bibitem{MeilijsonNadas} Meilijson, I. and N\'{a}das, A. (1979). Convex majorization with an application to the length of critical paths. {\em J. Appl. Prob.}, {\bf 16(3)}, 671--677.

\bibitem{Rost} R\"{o}st, H. (1981). Nonequilibrium behaviour of a many particle process: Density profile and local equilibria. {\em Zeitschrift f. Warsch. Verw. Gebiete}, {\bf 58.1}, 41--53.

\bibitem{Ruschendorf} R\"{u}schendorf, L. and Uckelmann, L. (2002). Variance minimization and random variables with constant sum. In Distributions with Given Marginals and Statistical Modelling, pp. 211–222. Dordrecht: Kluwer Academic Publishers. MR2058994

\bibitem{Sasamoto} Sasamoto, T. (2005). Spatial correlations of the 1D KPZ surface on a flat substrate.
{\em J. Phys. A}, {\bf 38}, 549–-556.

\bibitem{Shaked} Shaked, M. and Shanthikumar, J. G. (2007). {\em Stochastic orders}. Springer.

\bibitem{Wang} Wang, B. and Wang, R. (2011). The complete mixability and convex minimization problems with monotone marginal densities. {\em Journal of Multivariate Analysis}, {\bf 102 (10)},1344--1360.


\end{thebibliography}
\end{document}